\documentclass[3p]{elsarticle}

\usepackage{mathtools}
\usepackage{amsmath,amssymb}

\usepackage{latexsym}

\newtheorem{theorem}{Theorem}[section]
\newtheorem{lemma}[theorem]{Lemma}

\newtheorem{corollary}{Corollary}

\newdefinition{definition}{Definition}[section]
\newdefinition{conjecture}{Conjecture}[section]
\newdefinition{example}{Example}[section]
\newdefinition{remark}{Remark}
\newdefinition{note}{Note}
\newdefinition{case}{Case}

\newproof{proof}{Proof}

\usepackage{enumitem}

\usepackage{hyperref} 

\DeclareMathOperator{\im}{Im}
\DeclareMathOperator*{\spanof}{span}

\newcommand{\setof}[1]{\left\{{#1}\right\}}
\newcommand{\cset}[2]{\setof{#1\,:\,#2}}
\newcommand{\ivcc}[2]{\left[#1,#2\right]}

 \newcommand{\ivu}{\ivcc{0}{1}}

\newcommand{\abs}[1]{\left|#1\right|}
\newcommand{\norm}[1]{\left\Vert#1\right\Vert}

\newcommand{\normop}[1]{\left\Vert#1\right\Vert_{op}}

\newcommand{\spacecf}{C(\ivcc{0}{1})}

\newcommand{\ballo}[2]{B({#1},{#2})}
\newcommand{\ballc}[2]{\overline{\ballo{#1}{#2}}}
\newcommand{\uballo}{\ballo{0}{1}}
\newcommand{\uballc}{\ballc{0}{1}}

\newcommand{\Rr}{{\mathbb R}}
\newcommand{\Cc}{{\mathbb C}}

\newcommand{\boundop}[2]{\mathcal{B}(#1,#2)}

\newcommand{\ie}{i.\,e.}

\newcommand{\eg}{e.\,g.}

\journal{Journal of Approximation Theory}

\usepackage{fancyhdr}
\makeatletter
\def\ps@pprintTitle{%
 \let\@oddhead\@empty
 \let\@evenhead\@empty
 \def\@oddfoot{{\footnotesize\itshape \hfill Preprint, March 18, 2014}}%
 \let\@evenfoot\@oddfoot}
\makeatother

\begin{document}

\begin{frontmatter}
  \title{On the spectrum of positive finite-rank operators\\ with a partition of unity property}

  \author{Johannes~Nagler}
  \ead{johannes.nagler@uni-passau.de}

  \address{Fakult\"at f\"ur Informatik und Mathematik, Universit\"at Passau, Germany}

  \begin{abstract}
    We characterize the spectrum of positive linear operators $T:X \to Y$,
    where $X$ and $Y$ are complex Banach function spaces with unit $1$, 
    having finite rank and a partition of unity property.
    Then all the points in the spectrum are eigenvalues of $T$ and $\sigma_p(T) \subset \uballo \cup \setof{1}$.
    The main result is that $1$ is the only eigenvalue on the unit circle.
  \end{abstract}

\end{frontmatter}

\newcommand{\polkn}{e_{k}}
\newcommand{\funkn}{\alpha^*_{k}}
\newcommand{\poljn}[1]{e_{#1}}
\newcommand{\funjn}[1]{\alpha^*_{#1}}

We study positive linear operators that have finite-rank on some general 
infinite-dimensional complex Banach function spaces $X, Y$ with unit $1$. 
In addition, we assume that the positive finite-rank operator $T:X \to Y$ is unitary 
in the sense that $T1 = 1$.
This work generalized some results of the manuscript of Nagler, Cerejeiras, and Forster \cite{nagler2014}, 
where the spectral properties has been shown concretely for the Schoenberg operator 
in order to prove the limit of iterates and lower bounds. Since the result is established in a 
more general setting it is of interest of its own, as it is applicable to  positive
operators \eg,  on the space of continuous functions as well as on the space of integrable functions.
The results shown might serve as a starting point and can be used as in \cite{nagler2014} 
to prove the convergence of iterates by using spectral properties or to prove lower bounds
for the approximation error of other finite-rank operators.

\section{The spectrum of positive finite-rank operators}
We will discuss the positive finite-rank operator $T: X \to Y$ for $f \in X$ defined as
\begin{equation}
  \label{eq:discrete_operator}
  Tf = \sum_{k=1}^n \funkn(f)\cdot \polkn,\qquad f \in X,
\end{equation}
where $0 \leq \polkn \in X$ are positive functions of $X$ and $\funkn$ are  positive linear functionals with $\norm{\funkn} = \funkn(1) = 1$ for $k \in \setof{1,\ldots, n}$. 
In addition, let the basis functions $\polkn$ form a partition of unity, i.e., 
\begin{equation}
  \label{eq:partition_unity}
  \sum_{k=1}^n \polkn = 1.
\end{equation}
The following theorem is our main result.
\begin{theorem}[The spectrum of $T$]
  \label{thm:spectrum_of_finite-rank}
  The spectrum of the operator $T$, defined by \eqref{eq:discrete_operator}, consists only of the 
  point spectrum and is characterized by
  \begin{equation*}
    \sigma(T) = \sigma_p(T) \subset \uballo \cup \setof{1}.
  \end{equation*}
\end{theorem}
A more precise characterization of the spectrum gives the next corollary.
\begin{corollary}
  \label{cor:sumt_properties}
  The positive finite-rank operator $T \in \boundop{X}{Y}$ has the following properties:
  \begin{enumerate}
  \item $1 \in \ker(T - I)$, \ie, $1$ is an eigenvalue of $T$, and
  \item $\sigma_p(T) = \sigma(T) \subset \uballo \cup \setof{1}$.
  \end{enumerate}
  I.e., the only eigenvalue in the peripheral spectrum is $1$.
\end{corollary}
Before we prove these, we first give some examples where our main result can be applied.
\begin{example}[$\spacecf$ and point evaluations]
The Riesz representation theorem gives a characterisation of positive
linear functionals on $\spacecf$. Namely, for every positive linear functional 
$a^*: \spacecf \to \Cc$, there is a unique positive Radon measure $\nu$ such that
\begin{equation*}
  a^*(f) = \int_0^1 f \mathrm{d}\nu \qquad\text{for every } f \in \spacecf.
\end{equation*}

  A classical example of a positive linear functional on $\spacecf$ is
  the Dirac measure at a point $x \in \ivu$ defined for $f \in \spacecf$ by
  \begin{equation*}
    \delta_x(f) = f(x).
  \end{equation*}
  Given a partition $\Delta_n = \setof{x_k}_{k=1}^n$ of $\ivu$  satisfying
  \begin{equation*}
    0 = x_1 < x_2 < \ldots < x_n = 1,
  \end{equation*}
  then a popular choice for the functionals are $a^*_k = \delta_{x_k}$ for $k \in \setof{1,\ldots,n}$. 
  Then the positive finite-rank operator can be written for $x \in \ivu$ as
  \begin{equation*}
    Tf (x) = \sum_{k=1}^n f(x_j)\cdot \polkn(x), \qquad \polkn \in \spacecf.
  \end{equation*}
  Operators of this kind are often used to approximate continuous functions by only a finite number of 
  samples. Using \autoref{thm:spectrum_of_finite-rank} we obtain that $\sigma(T) \subset \uballo \cup \setof{1}$.  
\end{example}

\begin{example}[Kantorovi{\v c} operator on $L^1(\ivu)$]
  The Weierstrass approximation theorem says that every continuous
  function can be uniformly approximated by polynomials. An often used technique to prove this theorem are
  the classical Bernstein polynomials and the corresponding Bernstein operator. 
  This result can be transfered to the space of integrable functions using
  the so called Kantorovi{\v c} operators. In 1930, Kantorovi{\v c}~\cite{kantorovich1930} has defined
  a sequence of operators $K_n:L^1(\ivu) \to \spacecf$ as
  \begin{equation*}
    K_nf(x) = (n+1)\sum_{k=0}^n \binom{n}{k}x^k(1-x)^{n-k}\int_{\frac{k}{n+1}}^{\frac{k+1}{n+1}}f(t)\mathrm{d}t,\qquad
    f \in L^1(\ivu).
  \end{equation*}
  In fact, $K_n1 = 1$ holds and each of the operators $K_n$ is positive and has finite-rank.
  Using the main result, we can characterize the spectrum of these operators:
  $\sigma(K_n) \subset \uballo \cup \setof{1}$.
\end{example}

\section{Notation}
In the following let $X$ and $Y$ be Banach spaces with topological duals $X^*$ and $Y^*$.
We denote the space of bounded linear operators from $X$ to $Y$ by $\boundop{X}{Y}$ 
equipped with the usual operator norm $\normop{\cdot}$. 
With $I$ we denote the identity operator on $\boundop{X}{Y}$. 
For $T \in \boundop{X}{Y}$, we denote by $\sigma(T)$ the spectrum of $T$, 
\begin{equation*}
  \sigma(T) = \cset{\lambda \in \Cc}{T - \lambda I \text{ is not invertible}}.
\end{equation*}
By $\sigma_p(T)$, we denote the point spectrum of $T$, 
\begin{equation*}
  \sigma_p(T) = \cset{\lambda \in \Cc}{T - \lambda I \text{ is not one-to-one}},
\end{equation*}
which contains all the eigenvalues of $T$.
The open ball of radius $r > 0$ at the point $z \in \Cc$ in the complex plane will be 
denoted by $\ballo{z}{r} := \cset{\lambda \in \Cc}{\abs{\lambda - z} < r}$ and its closure by $\ballc{z}{r}$.
For $M \subset X$ we denote by
\begin{equation*}
  M^\bot = \cset{x^* \in X^*}{x^*(x) = 0 \text{ for every } x \in M} \subset X^*,
\end{equation*}
the \emph{annihilator} of $M$, whereas the \emph{pre-annihilator} of the set 
$\Lambda \subset X^*$ will be denoted by
\begin{equation*}
  \Lambda_\bot = \cset{x \in X}{x^*(x) = 0 \text{ for every } x^* \in \Lambda} \subset X.
\end{equation*}
The annihilator set $M^\bot$ contains all continuous linear functionals on $X$ that vanish on $M$, while
 $\Lambda_\bot$ is the subset of $X$ on which every bounded functional from $\Lambda$ is zero.
For properties of annihilators we refer to the book of Rudin~\cite{rudin1991}.

\section{Basic properties of positive finite-rank operators}
This section discusses properties that characterize the positive finite-rank operator $T$.
The next lemma states the positivity of $T$ and the ability to reconstruct constants.
\begin{lemma}
  The linear operator $T$, defined by \eqref{eq:discrete_operator}, is
  positive and reproduces constants.
\end{lemma}
\begin{proof}
  As the $\funkn$ are linear positive functionals and $\polkn \geq 0$, 
  we conclude for $f \in \spacecf$, $f \geq 0$,
  \begin{equation*}
    Tf = \sum_{k=1}^n \funkn(f)\polkn \geq 0.
  \end{equation*}
  And we obtain by applying the preconditions on $T$ that
  \begin{equation*}
    T1 = \sum_{k=1}^n \funkn(1) \polkn = \sum_{k=1}^n \polkn = 1.
  \end{equation*}
\end{proof}
\begin{lemma}
  The operator $T:X \to Y$ is bounded and $\normop{T} = 1$.
\end{lemma}
\begin{proof}
  Let $f \in X$ such that $\norm{f} = 1$. Then
  \begin{equation*}
    \norm{Tf} = \norm{\sum_{k=1}^n \funkn(f)\cdot \polkn} \leq \max_k \abs{\funkn(f)} \cdot \sum_{k=1}^n \norm{\polkn}              \leq \norm{f} \cdot \max_k \norm{\funkn} = 1,
  \end{equation*}
  where we used the partition of unity \eqref{eq:partition_unity} and 
  that $\norm{\funkn} = 1$. Using that $T1 = 1$, we conclude that $\normop{T} = 1$.
\end{proof}
Now we will proof that the operator $T$ is indeed a finite-rank
operator and give additional basic properties.
\begin{lemma}
  The linear operator $T$ has finite rank. 
  Thus, the operator $T$ is compact.
\end{lemma}
\begin{proof}
  As $\im(T f) = \spanof\cset{ \sum_{k=1}^n \funkn(f)\cdot \polkn}{f \in X}$, 
  the image can be written as a linear combination of the $n$ basis functions $\polkn$ and hence, 
  $\dim(\im(T f)) \leq n$.  Therefore, the linear operator $T$ has finite rank. 
  It follows that the operator $T$ is compact.
\end{proof}

\begin{theorem}
  The adjoint $T^*:Y^* \to X^*$ of $T$ is a finite-rank operator. 
  It is given for $x^* \in Y^*$ as
  \begin{equation*}
    T^*x^*(f) = \sum_{k=1}^n x^*(\polkn) \funkn(f), \qquad f \in X.
  \end{equation*}
\end{theorem}
\begin{proof}
  We calculate
  \begin{equation*}
    x^*(Tf) = x^*(\sum_{k=1}^n \funkn(f)\cdot \polkn)  
    = \sum_{k=1}^n \funkn(f) x^*(\polkn) 
    = T^*x^*(f).
  \end{equation*}  
\end{proof}

%
%
%

\section{Proof of the main result}
  Since $\normop{T} = 1$, the inequality
  \begin{equation*}
    \abs{\lambda} \leq \norm{T} = 1
  \end{equation*}
  holds for each $\lambda \in \sigma(T)$. Therefore, $\sigma(T) \subset \uballc$. 

  In the following, we show that $\sigma(T) \subset \uballo \cup
  \setof{1}$, i.e., if $\lambda \in \sigma(T)$ with $\abs{\lambda} =
  1$ then $\lambda = 1$ and all the spectral values are eigenvalues of $T$.

  The proof is organized as follows: first, we prove that $0 \in \sigma_p(T)$. 
  Then, we will show that $1 \in \sigma_p(T)$.  Finally, we consider eigenvalues 
  $\lambda \in \sigma_p(T) \setminus \setof{0,1}$ and we show that in this case $\abs{\lambda} < 1$ holds.

  Note that for compact operators it is known that every $\lambda \neq 0$ in the spectrum is 
  contained in the point spectrum.  This classical result is stated, e.g., in Rudin~\cite[Theorem~4.25]{rudin1991}.
  Therefore, if $0 \in \sigma_p(T)$, then it follows already that
  \begin{equation*}
    \sigma(T) = \sigma_p(T).
  \end{equation*}

  \begin{description}
  \item[Step 1:] 
    In order to prove that $0 \in \sigma_p(T)$ we show $\ker(T) \neq \setof{0}$.
    %
    %
    %
    %
    Using Rudin~\cite[Theorem 4.12]{rudin1991}, we obtain that $\ker(T) = \im(T^*)_\bot$.
    As $\im(T)$ is closed in $Y$, so is $\im(T^*)$ weak$^*$-closed in $X*$.
    Suppose now that $\ker(T) = \setof{0}$. It follows that
    $\im(T^*)_\bot = \setof{0}$ and therefore, $(\im(T^*)_\bot)^\bot = X^*$.
    This requires that $\im(T^*)$ is weak$^*$-dense in $X^*$. This gives a contradiction
    as $\im(T^*) = \spanof\setof{\funjn{1},\ldots,\funjn{n}}$ is weak$^*$-closed and 
    $X^* \neq \spanof\setof{\funjn{1},\ldots,\funjn{n}}$, because $X^*$ is infinite-dimensional. 
    We conclude that $\ker(T) \neq \setof{0}$ and the finite-rank operator $T$ is not one-to-one, \ie,
    %
    %
    $0 \in \sigma_p(T)$.


    \item[Step 2:] We have $1 \in \sigma(T)$, because of the partition of the unity property and the
    function $f(x) = 1$ is an eigenfunction of $T$ corresponding to the eigenvalue $1$.

    \item[Step 3:] Now we prove that for all the other eigenvalues $\lambda \in \sigma(T)$, we have
    \begin{equation*}
      \abs{\lambda} < 1.
    \end{equation*}
    Let $\lambda \in \sigma(T) \setminus \setof{0}$. As the
    operator maps continuous functions to finite dimensional space $\im(T)$, the
    eigenfunctions have to be in this space, too.  
    Let $p \in \im(T)$, $p = \sum_{k=1}^{n}c_j \polkn$, be such an
    eigenfunction for the eigenvalue $\lambda$.  Then
    \begin{alignat*}{2}
      && & T p = \lambda p\\
      &&\Longleftrightarrow\qquad &
      \sum_{k=1}^{n}\sum_{j=1}^{n}c_j \funkn(\poljn{j}) \polkn(x)
      = \lambda \sum_{k=1}^{n}c_j \polkn(x)\\
      &&\Longleftrightarrow\qquad &
      \sum_{k=1}^{n}\left[\sum_{j=1}^{n}c_j\funkn(\poljn{j}) -
        \lambda c_k\right]\polkn(x)
      = 0\\
      &&\Longleftrightarrow\qquad &
      \sum_{j=1}^{n}c_j\funkn(\poljn{j}) = \lambda c_k,\qquad\text{for all } k \in \setof{0,\ldots,n}.
    \end{alignat*}
    Thus, $\lambda \neq 0$ is an eigenvalue of the operator $T$, if
    and only if $\lambda$ is an eigenvalue of the matrix $M \in \Rr^{n \times n}$,
    \begin{equation*}
      M =
      \begin{pmatrix}
        \funjn{1}(\poljn{1}) & \funjn{1}(\poljn{2}) & \cdots & \funjn{1}(\poljn{n})\\
        \funjn{2}(\poljn{1}) & \funjn{2}(\poljn{2}) & \cdots & \funjn{2}(\poljn{n})\\
        \vdots & & &\\
        \funjn{n}(\poljn{1}) & \funjn{n}(\poljn{2}) & \cdots & \funjn{n}(\poljn{n})\\
      \end{pmatrix}.
    \end{equation*}
    This matrix $M$ is nonnegative as $\poljn{k} \geq 0$ and $\funjn{k}$ are positive linear functionals.
    Moreover, every row sums up to one because of the partition of unity property.
    To see this, we calculate for some fixed row $k \in \setof{1,\ldots,n}$
    \begin{equation}
      \label{eq:sumt_row_sum_one}
      \sum_{i=1}^n \funjn{k}(\poljn{j}) = \funjn{k}( \sum_{j=1}^n \poljn{j} ) = \funjn{k}(1) = 1.
    \end{equation}

    Therefore, by the Theorem of Gershgorin \cite{Gershgorin:1931}, we have that
    the eigenvalues are contained in the union of circles
    \begin{equation*}
      \lambda \in \bigcup_{j=-k}^{n-1} D_j,
    \end{equation*}
    with 
    \begin{equation*}
      D_k = \cset{\lambda \in \Cc}{\abs{\lambda - \funjn{k}(\poljn{k})}
        \leq \sum_{j=1, j\neq k}^{n} \funjn{k}(\poljn{j})}.
    \end{equation*}
    Using the partition of unity property \eqref{eq:partition_unity}, it follows that
    \begin{equation*}
      \bigcup_{k=1}^{n} D_k \cap \cset{\lambda \in \Cc}{\abs{\lambda} = 1}   = \setof{1}.
    \end{equation*}  
    Finally, we obtain $\sigma_p(T) = \sigma(T) \subset \uballo \cup \setof{1}$.
  \end{description}
\hfill $\Box$

\bibliographystyle{elsarticle-num}
\bibliography{references}

\end{document}